\documentclass[12pt, reqno]{amsart}
\usepackage{amsmath,amssymb}

\theoremstyle{plain}
\newtheorem{thm}{Theorem}[section]
\newtheorem*{main}{Main~Theorem}

\newtheorem{remark}{Remark}  

\newtheorem{corollary}[thm]{Corollary}
\newtheorem{lemma}[thm]{Lemma}

\renewcommand{\div}{\operatorname{div}}
\newcommand{\trace}{\operatorname{tr}}
\newcommand{\D}{\displaystyle}
\newcommand{\ip}[2]{\ensuremath{\langle #1 , #2 \rangle}}
\newcommand{\tp}{\texttt{p}}
\newcommand{\ri}{\ensuremath{\rho_i (x_1,x_2,\ldots,x_{i-1})}}
\newcommand{\rj}{\ensuremath{\rho_j (x_1,x_2,\ldots,x_{j-1})}}

\theoremstyle{definition}
\newtheorem{definition}{Definition}

\theoremstyle{remark}

\numberwithin{equation}{section}
\begin{document}
\title{The $\infty(x)$-equation in Grushin-type Spaces}
\author{Thomas Bieske}
\address{Department of Mathematics \\
University of South Florida\\
Tampa, FL 33620, USA} 
\email{tbieske@mail.usf.edu}
\date{September 11, 2015}
\begin{abstract}
We employ Grushin jets which are adapted to the geometry of Grushin-type spaces to obtain the existence-uniqueness of viscosity solutions to the $\infty(x)$-Laplace equation in Grushin-type spaces. Due to the differences between Euclidean jets and Grushin jets, the Euclidean method of proof is not valid in this environment. 
\end{abstract}

\maketitle

\section{Introduction}  
Recently, the $\tp(x)$-Laplace equation and its limit equation, the $\infty(x)$-Laplace equation, have been the focus of much attention as a tool for exploring applications such as image restoration \cite{CLR} and electrorheological fluid flow \cite{R}. Linqvist and Luukari \cite{LL} recently proved existence-uniqueness of viscosity solutions to the $\infty(x)$-Laplace equation in (Euclidean) $\mathbb{R}^n$.  However, this proof is not valid in general Carnot- \\ Carath\'{e}odory spaces, such as Grushin-type spaces, because it relies on two important Euclidean properties, namely that the so-called viscosity penalty function is the square of the intrinsic distance and that the two first-order jet elements derived from the penalty function are equal. (These two phenomena are discussed more below.) The main result of this paper is that the lack of these phenomena in Grushin-type spaces can be overcome to produce existence-uniqueness of viscosity solutions in this environment. In particular, we prove the following theorem:
\begin{main}
    Let $\Omega$ be a bounded domain in the Grushin-type space $G_n$ and let $f:\partial \Omega \to \mathbb{R}$ be a (Grushin) Lipschitz function.
    Then the Dirichlet problem 
    \begin{eqnarray*}
\left\{ \begin{array}{cc}
-\Delta_{\infty(x)} u  =  0  & \textmd{in}\  \Omega \\
u = f & \textmd{on}\ \partial \Omega
\end{array} \right.
\end{eqnarray*}
has a unique viscosity solution $u$. 
\end{main}

In Section 2, we review the geometry of Grushin-type spaces and definitions of various viscosity solutions. Section 3 collects all the Grushin tools we will be using in our proof of existence-uniqueness, found in Section 4. Section 5 details some further properties of the viscosity solutions. 

\section{Grushin-type spaces}
\subsection{The Environment}
We begin by constructing the Grushin-type spaces.  We consider $\mathbb{R}^n$
with
coordinates $(x_1,x_2,\ldots,x_n)$ and the vector fields 
$$X_i = \ri \frac{\partial}{\partial x_i}$$ for $i=2,3,\ldots, n$ where $\ri$ is a  (possibly constant) polynomial.  We decree that $\rho_1 \equiv 1$ so that
$$X_1=\frac{\partial}{\partial x_1}.$$
A quick calculation shows that when $i < j$, the Lie bracket is given by
\begin{equation}\label{bracket}
X_{ij} \equiv [X_i,X_j]= \ri \frac{\partial \rj}{\partial x_i} 
\frac{\partial }{\partial x_j}.
\end{equation}
Because the $\rho_i$'s are
polynomials, at each point there is a finite number of iterations of the
Lie bracket so that $\frac{\partial}{\partial x_i}$ has a non-zero coefficient.
This is easily seen for $X_1$ and $X_2$, and the result is obtained
inductively for $X_i$.  (It is noted that the number of iterations
necessary is a function of the point.)  
Thus, H\"{o}rmander's condition is satisfied by these
 vector fields.  
Endowing $\mathbb{R}^n$ with an inner product (singular where the polynomials vanish) so
that the $X_i's$ are orthonormal produces a manifold that we shall call $g_n$.  This
is the tangent space to a  generalized Grushin-type space $G_n$.  Points in $G_n$
will also be denoted by $x=(x_1,x_2,\ldots, x_n)$ with a fixed point
denoted
$x_0=(x^0_1,x^0_2,\ldots,x^0_n)$.  

Even though $G_n$ is not a group, it is a metric space with the natural metric
being the Carnot-Carath\'{e}odory distance, which
is defined for the points $x$ and $y$ as follows:
 \begin{equation*}
d_C(x,y)= \inf_{\Gamma} \int_{0}^{1} \| \gamma '(t) \| dt 
\end{equation*}
\noindent  where the set $ \Gamma $
is the set of all  curves $ \gamma $ such that $ \gamma (0) = x,
 \gamma (1) = y $ and $\gamma '(t) $ is in
span$\{\{X_i(\gamma(t))\}_{i=1}^n\}$. By Chow's theorem (see, for example,
\cite{BR:SRG}) any two points
 can be connected by such a curve, which means $ d_C(x,y) $ is an honest
metric.  Using this metric, we can define a
Carnot-Carath\'{e}odory ball of radius $r$ centered at a point $x_0$ by
$$B_r=B(x_0,r)=\{p\in G_n : d_C(x,x_0) < r\}$$ and similarly, we shall denote a
bounded domain in $G_n$ by $\Omega$.
The Carnot-Carath\'{e}odory metric behaves differently when the
polynomials $\ri$ vanish.  Fixing a point $x_0$, consider the  $n$-tuple
$r_{x_0}=(r^1_{x_0},r^2_{x_0},\ldots,r^n_{x_0})$ where $r^i_{x_0}$ is the minimal length of the 
 Lie bracket iteration 
 required to produce $$[X_{j_1},[X_{j_2},[\cdots[X_{j_{r^i_{x_0}}},X_i]\cdots](x_0) \neq 0.$$ 
 Note that even though the minimal length is unique, the iteration used to obtain that minimum is 
 not unique.  Note also that $$\rho_i(x_1^0, x_2^0, x_3^0,\ldots, x_{i-1}^0)
 \neq 0 \leftrightarrow r^i_{x_0}=0.$$
Using Theorem 7.34 from \cite{BR:SRG} we obtain the local
estimate at $x_0$
\begin{equation} \label{distest}
d_C(x_0,x)  \sim  \sum _{i=1}^n |x_i-x_i^0|^\frac{1}{1+r^i_{x_0}} . 
\end{equation}
 
Given a smooth function $f$ on $G_n$, we define the horizontal gradient of $f$ as 
$$\nabla_0f(x) = (X_1f(x),X_2f(x),\ldots, X_nf(x))$$
and the symmetrized second order (horizontal)  derivative matrix by 
$$((D^2f(x))^{\star})_{ij} = \frac{1}{2} (X_iX_jf(x)+X_jX_if(x))$$ 
for $i,j=1,2,\ldots n$.  We can then define function spaces $C^k$ and the Sobolev spaces $W^{1,\tp}$, etc with respect to these vector fields in the usual way. 

We may also define the $\infty$-Laplace operator
\begin{equation*}
\Delta_{\infty} u  =
 \ip{((D^2u(x))^{\star}) \nabla_0 u}{\nabla_0 u}.
\end{equation*}
This operator is the ``limit" operator of the $\tp$-Laplace operator (for $2< p<\infty$), which is given by 
\begin{eqnarray*}
\Delta_{\tp} u & = &
 \|\nabla_0 u\|^{\tp-2}
\Delta u +(\tp-2) \|\nabla_0 u\|^{\tp-4}
 \Delta_{\infty} u\\
& = &\div_{G} \ (\|\nabla_0 u\|^{\tp-2} \nabla_0 u)
\end{eqnarray*}
where the divergence is taken with respect to the Grushin vector fields. 

Following \cite{LL}, we generalize these operators by replacing the constant $\tp$ with an appropriate function $\tp(x)\in C^1\cap W^{1,\infty}$ and scalar $k>1$ to obtain the $\tp(x)$-Laplace operator
\begin{eqnarray*}
\Delta_{\tp(x)} u & = &
 \|\nabla_0 u\|^{k\tp(x)-2}
\Delta u +(k\tp(x)-2) \|\nabla_0 u\|^{k\tp(x)-4}
 \Delta_{\infty} u\\
 & & \mbox{}+\|\nabla_0 u\|^{k\tp(x)-2}\ip{\nabla_0 u}{\nabla_0k\tp(x)}\ln \|\nabla_0 u\| \\
& = &\div_{G} \ (\|\nabla_0 u\|^{k\tp(x)-2} \nabla_0 u).\end{eqnarray*}

The corresponding equation $\Delta_{\tp(x)} u=0$ is the Euler-Lagrange equation associated to the energy functional $$\Bigg(\int_\Omega \frac{\|\nabla_0 u\|^{k\tp(x)}}{k\tp(x)}\;dx\Bigg)^\frac{1}{k}_.$$ Allowing $k\to\infty$, one has the tool for analysis of the extremal problem $$\min_u\max_x \|\nabla_0 u\|^{\tp(x)}.$$

Letting $k\to\infty$, we have $\Delta_{\tp(x)} u \to \Delta_{\infty(x)} u$ where 
\begin{equation*}
\Delta_{\infty(x)} u =
\Delta_{\infty} u + \|\nabla_0 u\|^{2}\ip{\nabla_0 u}{\nabla_0\ln\tp(x)}\ln \|\nabla_0 u\|.
\end{equation*}  
For more details concerning the geometry of Grushin-type spaces, the interested reader is directed to \cite{B:GC, CC} and the references therein.

\subsection{Viscosity Solutions}
Because we will be considering viscosity solutions, we will recall the main definitions and properties. We begin with the Grushin jets $J^{2,+}$ and $J^{2,-}$. 
(See \cite{B:GS,B:GC} for a more complete analysis of such jets.) 
\begin{definition}
Let $u$ be an upper semi-continuous function. Consider the set
\begin{eqnarray*}
K_{\mathfrak{X}}^{2,+}u(x)  = \bigg\{
\varphi \in C^2\ \textmd{in a neighborhood of}\ x,  \varphi(x) = u(x),\\
\varphi(y) \geq u(y), \ y\neq x \ \text{in a neighborhood of}\ x\bigg\}.
\end{eqnarray*}
Each function $\varphi\in K_{\mathfrak{X}}^{2,+}u(x)$ determines a vector-matrix pair $(\eta,X)$ via the relations
\begin{equation*}\begin{array}{rcl}
\eta&  = & \big(X_1\varphi(x),X_2\varphi(x),\ldots, X_n\varphi(x)\big)\\
 X_{ij}&=& \frac{1}{2} \big(X_i(X_j(\varphi))(x)+X_j(X_i(\varphi))(x)\big).
\end{array}
\end{equation*}
We then define the \emph{second order superjet of $u$ at $x$} by $$J^{2,+}u(x)=\{(\eta,X):\varphi \in K^{2,+}u(x)\},$$
the \emph{second order subjet of $u$ at $x$} by $$J^{2,-}u(x)=-J^{2,+}(-u)(x)$$ and the set-theoretic closure 
\begin{eqnarray*}
\overline{J}^{2,+}u(x)=\{(\eta,X): & \exists \{x_n,\eta_n,X_n\}_{n\in \mathbb{N}}\ \textmd{with}\ (\eta_n,X_n)\in J^{2,+}u(x_n)\\ & \textmd{and}\ (x_n,u(x_n),\eta_n,X_n)\to(x,u(x),\eta,X)\}.
\end{eqnarray*}
There is an analogous definition for $\overline{J}^{2,-}v(x)$. 
\end{definition}

We then use these Grushin jets to define viscosity $\infty(x)$-harmonic functions as follows:
\begin{definition}
A lower semi-continuous function  $v$  
  is \textbf{viscosity $\infty(x)$-superharmonic} in a bounded domain $\Omega$ if $v \not \equiv \infty$ in each component of $\Omega$ and for all $x_0 \in \Omega$,
whenever $(\xi, \mathcal{Y}) \in \overline{J}^{2,-} v(x_0)$, we have
$$-\Big(\ip{\mathcal{Y}\xi}{\xi} + \|\xi\|^{2}\ip{\xi}{\nabla_0\ln\tp(x)}\ln \|\xi\| \Big)\geq 0.$$
An upper semi-continuous function  $u$  
  is  \textbf{viscosity $\infty(x)$-subharmonic} in a bounded domain $\Omega$ if $u \not \equiv -\infty$ in each component of $\Omega$ and for all $x_0 \in \Omega$,
whenever $(\eta,  \mathcal{X}) \in \overline{J}^{2,+} u(x_0)$, we have
$$-\Big(\ip{\mathcal{X}\eta}{\eta} + \|\eta\|^{2}\ip{\eta}{\nabla_0\ln\tp(x)}\ln \|\eta\| \Big)\leq 0.$$
A function is \textbf{viscosity $\infty(x)$-harmonic} if it is both viscosity $\infty(x)$-subharmonic and viscosity $\infty(x)$-superharmonic.
\end{definition}

Similarly, we have the following definition concerning $\Delta_{\tp(x)} u$.
\begin{definition}
A lower semi-continuous function  $v$  
  is \textbf{viscosity $\tp(x)$-superharmonic} in a bounded domain $\Omega$ if $v \not \equiv \infty$ in each component of $\Omega$ and for all $x_0 \in \Omega$,
whenever $(\xi, \mathcal{Y}) \in \overline{J}^{2,-} v(x_0)$, we have
$$-\Big(\|\xi\|^{k\tp(x)-2}\trace \mathcal{Y} +(k\tp(x)-2) \|\xi\|^{k\tp(x)-4}
 \ip{\mathcal{Y}\xi}{\xi}+\|\xi\|^{k\tp(x)-2}\ip{\xi}{\nabla_0k\tp(x)}\ln \|\xi\| \Big)\geq 0.$$

An upper semi-continuous function  $u$ is  \textbf{viscosity $\tp(x)$-subharmonic} in a bounded domain $\Omega$ if $u \not \equiv -\infty$ in each component of $\Omega$ and for all $x_0 \in \Omega$,
whenever $(\eta,  \mathcal{X}) \in \overline{J}^{2,+} u(x_0)$, we have
$$-\Big(\|\eta\|^{k\tp(x)-2}\trace \mathcal{X} +(k\tp(x)-2) \|\eta\|^{k\tp(x)-4}
 \ip{\mathcal{X}\eta}{\eta}+\|\eta\|^{k\tp(x)-2}\ip{\eta}{\nabla_0k\tp(x)}\ln \|\eta\|\Big)\leq 0.$$
A function is \textbf{viscosity $\tp(x)$-harmonic} if it is both viscosity $\tp(x)$-subharmonic and viscosity $\tp(x)$-superharmonic.
\end{definition}
\begin{remark}
In the above definitions, we may replace the right-hand side of each inequality by an arbitrary function. In that case, we use the term viscosity $\infty(x)$-subsolution, etc. 
\end{remark}
\section{Key Grushin Tools and Results}
The Euclidean proof of the Main Theorem relies on two important facts, neither of which hold in Grushin-type spaces. The first is that the square of the intrinsic distance is a valid viscosity ``penalty function". By Equation \eqref{distest}, the square of the Grushin distance may not be sufficiently smooth. This can be rectified by choosing a higher power of the Grushin distance, so that the penalty function is smooth. This choice, however, is at odds with the reason for choosing the square of the distance function: the derivative of the penalty function must be comparable to the intrinsic distance so that we can exploit the Lipschitz property to obtain estimates that are controllable. 

We now present the Grushin tools and results we will need to overcome this issue. 
The first is the Iterated Maximum Principle. 
\begin{lemma}[Iterated Maximum Principle] \cite{B:GC} \label{malpha}
Let $u$ be an upper-semicontinuous function in a domain $\Omega$  and $v$
be a lower-semicontinuous function in $\Omega$.

Let $$\sup_{\overline{\Omega}} (u(x)-v(x))=u(x_0)-v(x_0)>0$$ occur in the interior of $\Omega$. 
Consider a real vector $\vec{\alpha} = (\alpha_1,\alpha_2,\ldots, \alpha_n)$ with non-negative components and the points $x$ and $y$ with coordinates 
$x=(x_1,x_2,\ldots, x_n)$ and  $y=(y_1,y_2,\ldots,y_n)$. We define the following functions for $j=1,2,3,\ldots, n$:
\begin{equation*}
\varphi_{\alpha_j,\alpha_{j+1},\ldots,\alpha_n}(x,y) = \sum_{i=j}^n \frac{1}{2}\alpha_i(x_i-y_i)^2.
\end{equation*}
Using these functions and upper-semicontinuity on a compact set, we can consider the following well-defined functions for $j=1,2,3,\ldots, n$:
\begin{eqnarray*}
\lefteqn{M_{\alpha_j,\alpha_{j+1},\ldots,\alpha_n} =}\\ && \sup_{\overline{\Omega} \times \overline{\Omega}}(u(x)-v(y)-\varphi_{\alpha_j,\alpha_{j+1},\ldots,\alpha_n}(x,y):x_k=y_k \ \textmd{for}\ k=1,2,\ldots j-1) \\
 && =u(x_{\alpha_j,\alpha_{j+1},\ldots,\alpha_n})-v(y_{\alpha_j,\alpha_{j+1},\ldots,\alpha_n})\\
& & \mbox{ }-\varphi_{\alpha_j,\alpha_{j+1},\ldots,
\alpha_n}(x_{\alpha_j,\alpha_{j+1},\ldots,\alpha_n},y_{\alpha_j,\alpha_{j+1},\ldots,\alpha_n}).
\end{eqnarray*}
We then have 
\begin{eqnarray*}
\lim_{\alpha_n \rightarrow \infty}\lim_{\alpha_{n-1} \rightarrow \infty}\cdots 
\lim_{\alpha_2 \rightarrow \infty}\lim_{\alpha_1 \rightarrow \infty} \varphi_{\alpha_1,\alpha_2,\ldots,\alpha_n}(x_{\alpha_1,\alpha_2,\ldots,\alpha_n},y_{\alpha_1,\alpha_2,\ldots,\alpha_n}) = 0\\
\textmd{and}\ \ 
\lim_{\alpha_n \rightarrow \infty}\lim_{\alpha_{n-1} \rightarrow \infty}\cdots \lim_{\alpha_2 \rightarrow \infty}
\lim_{\alpha_1 \rightarrow \infty} M_{\alpha_1,\alpha_2,\ldots,\alpha_n} = u(x_0)-v(x_0).
\end{eqnarray*}
\end{lemma}
\begin{corollary}\label{impcor}
Under the hypotheses of Lemma \ref{malpha}, all iterated limits exist and the full limit exists and equals the common value of the iterated limits.  That is, 
\begin{equation*}
\lim_{\alpha_1,\alpha_2,\ldots,\alpha_n \rightarrow \infty}
M_{\alpha_1,\alpha_2,\ldots,\alpha_n} =
\sup_{\overline{\Omega}} (u(x)-v(x))=u(x_0)-v(x_0)
\end{equation*}
and
\begin{equation*}
\lim_{\alpha_1,\alpha_2,\ldots,\alpha_n \rightarrow \infty}
\varphi_{\alpha_1,\alpha_2,\ldots,\alpha_n}
(x_{\alpha_1,\alpha_2,\ldots,\alpha_n},y_{\alpha_1,\alpha_2,\ldots,\alpha_n})=0.
\end{equation*}
\end{corollary}
\begin{remark}\label{remarkcomp}
As a consequence of the Iterated Maximum Principle, its proof, and Corollary \ref{impcor}, if we denote the 
points $x_{\alpha_1,\alpha_2,\ldots,\alpha_n}$ and $y_{\alpha_1,\alpha_2,\ldots,\alpha_n}$ by $x^{\vec{\alpha}}$ and $y^{\vec{\alpha}}$, respectively, then we have 
\begin{eqnarray*}
\lim_{\alpha_1 \rightarrow \infty}x^{\vec{\alpha}} = x_{\alpha_2,\alpha_{3},\ldots,\alpha_n} & = &(x_1^0,x^{\vec{\alpha}}_2,x^{\vec{\alpha}}_3, x^{\vec{\alpha}}_4, \ldots, x^{\vec{\alpha}}_n)\\
\lim_{\alpha_2 \rightarrow \infty}\lim_{\alpha_1 \rightarrow \infty}x^{\vec{\alpha}} = x_{\alpha_3,\alpha_{4},\ldots,\alpha_n} & = & (x_1^0, x^0_2, x^{\vec{\alpha}}_3,x^{\vec{\alpha}}_4, \ldots, x^{\vec{\alpha}}_n)\\
\vdots \\
\textmd{and\ \ }\lim_{\alpha_n \rightarrow \infty}\lim_{\alpha_{n-1} \rightarrow \infty}\cdots\lim_{\alpha_2 \rightarrow \infty}\lim_{\alpha_1 \rightarrow \infty}x^{\vec{\alpha}} = 
x_0 & = & (x_1^0, x^0_2, x^0_3, \ldots, x^0_n). 
\end{eqnarray*}
Similarly, 
\begin{eqnarray*}
\lim_{\alpha_1 \rightarrow \infty}y^{\vec{\alpha}} = y_{\alpha_2,\alpha_{3},\ldots,\alpha_n} & = & (x_1^0,y^{\vec{\alpha}}_2,y^{\vec{\alpha}}_3, y^{\vec{\alpha}}_4, \ldots, y^{\vec{\alpha}}_n)\\
\lim_{\alpha_2 \rightarrow \infty}\lim_{\alpha_1 \rightarrow \infty}y^{\vec{\alpha}} = y_{\alpha_3,\alpha_{4},\ldots,\alpha_n} & = & (x_1^0, x^0_2, y^{\vec{\alpha}}_3,y^{\vec{\alpha}}_4, \ldots, y^{\vec{\alpha}}_n)\\
\vdots \\
\textmd{and\ \ }\lim_{\alpha_n \rightarrow \infty}\lim_{\alpha_{n-1} \rightarrow \infty}\cdots\lim_{\alpha_2 \rightarrow \infty}\lim_{\alpha_1 \rightarrow \infty}y^{\vec{\alpha}} = 
x_0 & = & (x_1^0, x^0_2, x^0_3, \ldots, x^0_n). 
\end{eqnarray*}
\end{remark}
The importance of the Iterated Maximum Principle is that it will allow us to isolate the vector fields by direction. In the Euclidean case, it is sufficient to use the Lipschitz property and then take the full limit, which is independent of direction. This approach is incompatible with the fact that the Grushin distance estimates vary at each point.  Thus, we must use the Lipschitz property with the directions independently. This will allow us to overcome the first technical challenge. In order to do this, we need the following lemma and corollary. 
\begin{lemma}\cite[Lemma 3.3]{B:GC}\label{lip}
Assume the hypotheses of the Iterated Maximum Principle (Lemma \ref{malpha}) and the notation of Remark \ref{remarkcomp}.  Suppose that at least one of $u$ or $v$ is (Grushin) Lipschitz. 
 Define the point $(x\diamond_i y)$ by $$(x\diamond_i y)=(x_1,x_2,\ldots,x_{i-1},y_i,x_{i+1},\ldots,x_n).$$ That is, $(x\diamond_i y)$ coincides with $y$ in the $i$-th coordinate and coincides with $x$ elsewhere. Then there is a finite positive constant $K$ so that 
$$\alpha_i(x_i^{\vec{\alpha}}-y_i^{\vec{\alpha}})^2 \leq K d_C((x^{\vec{\alpha}}\diamond_i y^{\vec{\alpha}}),x^{\vec{\alpha}}).$$
\end{lemma}
The following corollary follows immediately from Equation \eqref{distest}. 
\begin{corollary}\label{lipcor}
When $\rho_i(x_1^\alpha,x_2^\alpha,\ldots,x_{i-1}^\alpha)\neq 0$, locally, we have 
$$\alpha_i(x_i^{\vec{\alpha}}-y_i^{\vec{\alpha}})^2 \leq K d_C((x^{\vec{\alpha}}\diamond_i y^{\vec{\alpha}}),x^{\vec{\alpha}})=C|x_i^{\vec{\alpha}}-y_i^{\vec{\alpha}}|.$$
\end{corollary}

The second important fact used in the Euclidean proof is that the first-order jet elements of a viscosity $\infty(x)$-superharmonic and viscosity $\infty(x)$-subharmonic are identical. Because the polynomials are non-constant, this is not the case in the Grushin environment. Using Theorem 3.4 of \cite{B:MP} (or the Main Lemma of \cite{B:GS}), we have for an upper semicontinuous function $u$ and a lower semicontinuous function $v$,
\begin{equation*}
(\Upsilon_{x^{\vec{\alpha}}}, \mathcal{X}^{\vec{\alpha}})\in \overline{J}^{2,+}u(x^{\vec{\alpha}})\ \textmd{and}\ \ (\Upsilon_{y^{\vec{\alpha}}}, \mathcal{Y}^{\vec{\alpha}})\in \overline{J}^{2,-}v(y^{\vec{\alpha}}) 
\end{equation*} 
where $(x^{\vec{\alpha}},y^{\vec{\alpha}})$ are the points associated with $M_{\alpha_1,\alpha_{2},\ldots,\alpha_n}$ (from Remark \ref{remarkcomp}) and 
\begin{eqnarray}
(\Upsilon_{x^{\vec{\alpha}}})_i & = & \rho_i(x^{\vec{\alpha}}_1, x^{\vec{\alpha}}_2,\ldots, x^{\vec{\alpha}}_{i-1})\alpha_i(x^{\vec{\alpha}}_i-y^{\vec{\alpha}}_i) \label{jetvecdef}\\
\textmd{and}\ \ (\Upsilon_{y^{\vec{\alpha}}})_i & = & \rho_i(y^{\vec{\alpha}}_1, y^{\vec{\alpha}}_2,\ldots, y^{\vec{\alpha}}_{i-1}))\alpha_i(x^{\vec{\alpha}}_i-y^{\vec{\alpha}}_i).\nonumber
\end{eqnarray}
We see that these vectors are, in general, not equal. 
We will overcome this challenge by producing more complicated estimates that are still controllable. 

\section{Existence-Uniqueness of $\infty(x)$-harmonic functions}
Let $\Omega$ be a bounded domain in $G_n$ and $f:\partial \Omega\to\mathbb{R}$ be a (Grushin) Lipschitz function. 

We will first establish the existence of $\infty(x)$-harmonic functions using Jensen's  auxiliary equations \cite{Je:ULE}: 
\begin{equation*}
\min \{\|\nabla_0 u\|^2 - \varepsilon, - \Delta_{\infty(x)} u \}  =  0 
\ \ \textmd{and}\ \ \max \{\varepsilon -\|\nabla_0 u\|^2 , - \Delta_{\infty(x)} u \} = 0
\end{equation*}
for a real parameter $\varepsilon>0$ . 
The procedure for existence of viscosity solutions to these equations (and viscosity $\infty(x)$-harmonic functions)  is identical to \cite[Section 4]{B:HG} and \cite[Section 2]{LL}, up to the obvious modifications. For completeness, we state the steps as one theorem and omit the proofs. 
\begin{thm}\cite{LL, B:HG}\label{exist}
We have the following results:
\begin{enumerate}
\item Let $\varepsilon\in\mathbb{R}$. If $u_k$ is a continuous potential-theoretic weak sub-(super-)solution with $u\in W^{1,k\tp(x)}(\Omega)$ to:
   \begin{eqnarray*}
\left\{ \begin{array}{cc}
-\Delta_{k\tp(x)} u_k  =  \varepsilon^{k\tp(x)-1}  & \textmd{in}\  \Omega \\
u = f & \textmd{on}\ \partial \Omega
\end{array} \right.
\end{eqnarray*}
then it is a viscosity sub-(super-)solution.
\item Letting $k\to\infty$, we have $u_k\to u_\infty$ uniformly (possibly up to a subsequence) in $\Omega$ with $u_\infty\in W^{1,\infty}(\Omega)\cap C(\overline{\Omega})$.
\item The function $u_\infty$ is a viscosity solution to 
\begin{eqnarray*}
\min \{\|\nabla_0 u_\infty\|^2 - \varepsilon, - \Delta_{\infty(x)} u_\infty \} =0 & \textmd{when} & \varepsilon>0 \\
\max \{\varepsilon -\|\nabla_0 u_\infty\|^2 , - \Delta_{\infty(x)} u_\infty \}  =  0 & \textmd{when} & \varepsilon<0 \\
- \Delta_{\infty(x)} u_\infty =0  & \textmd{when} & \varepsilon=0. 
\end{eqnarray*}  
\end{enumerate}
\end{thm}

In light of \cite[Lemma 5.6]{B:HG} and \cite[Lemma 2.2]{LL}, the Main Theorem follows from showing the uniqueness of viscosity solutions to the Jensen auxiliary equations. We will establish this result.

\begin{thm}\label{compinf}
Let $v=u_\infty$ be the viscosity solution from Theorem \ref{exist} to 
\begin{equation}\label{minequ}
\min \{\|\nabla_0 u\|^2 - \varepsilon, - \Delta_{\infty(x)} u \} =0
\end{equation}
in a bounded domain $\Omega$. If $u$ is an upper semi-continuous viscosity subsolution to Equation \eqref{minequ} in $\Omega$ so that $u\leq v$ on $\partial \Omega$, then $u\leq v$ in $\Omega$. 
\end{thm}
\begin{proof}
Following \cite[Lemma 3.1]{LL} and \cite[Theorem 5.3]{B:HG}, we may assume WLOG that $v$ is a strict viscosity supersolution. Suppose $$\sup_{\Omega} (u-v)>0$$ and let $$\varphi_{\alpha_1,\alpha_{2},\ldots,\alpha_n}(x,y) = \sum_{i=1}^n \frac{1}{2}\alpha_i(x_i-y_i)^2.$$ be the function in the Iterated Maximum Principle (Theorem \ref{malpha}). By \cite[Theorem 3.2]{B:MP}, we have vectors $\Upsilon_{x^{\vec{\alpha}}}, \Upsilon_{y^{\vec{\alpha}}}$ and symmetric matrices $\mathcal{X}^{\vec{\alpha}}, \mathcal{Y}^{\vec{\alpha}}$ so that 
$$(\Upsilon_{x^{\vec{\alpha}}}, \mathcal{X}^{\vec{\alpha}})\in \overline{J}^{2,+}u(x^{\vec{\alpha}})\ \ \textmd{and}\ \ 
(\Upsilon_{y^{\vec{\alpha}}, \mathcal{Y}^{\vec{\alpha}}})\in \overline{J}^{2,-}v(y^{\vec{\alpha}}).$$  The vectors are explicitly given by Equation \eqref{jetvecdef}.

Since $u$ is a viscosity subsolution and $v$ a strict viscosity supersolution, we have, for some $\mu>0$, 
\begin{eqnarray*}
0 & \geq & \min \{\|\Upsilon_{x^{\vec{\alpha}}}\|^2 - \varepsilon, 
-\ip{\mathcal{X}^{\vec{\alpha}}\Upsilon_{x^{\vec{\alpha}}}}{\Upsilon_{x^{\vec{\alpha}}}} - \|\Upsilon_{x^{\vec{\alpha}}}\|^{2}\ip{\Upsilon_{x^{\vec{\alpha}}}}{\nabla_0\ln\tp(x^{\vec{\alpha}})}\ln \|\Upsilon_{x^{\vec{\alpha}}}\|\}\\
0<\mu & \leq &  \min \{\|\Upsilon_{y^{\vec{\alpha}}}\|^2 - \varepsilon, 
-\ip{\mathcal{Y}^{\vec{\alpha}}\Upsilon_{y^{\vec{\alpha}}}}{\Upsilon_{y^{\vec{\alpha}}}} - \|\Upsilon_{y^{\vec{\alpha}}}\|^{2}\ip{\Upsilon_{y^{\vec{\alpha}}}}{\nabla_0\ln\tp(y^{\vec{\alpha}})}\ln \|\Upsilon_{y^{\vec{\alpha}}}\|\}.
\end{eqnarray*}

Subtracting these equations, we obtain
\begin{eqnarray}
0<\mu & \leq & \max \{\|\Upsilon_{y^{\vec{\alpha}}}\|^2 - \|\Upsilon_{x^{\vec{\alpha}}}\|^2, 
\ip{\mathcal{X}^{\vec{\alpha}}\Upsilon_{x^{\vec{\alpha}}}}{\Upsilon_{x^{\vec{\alpha}}}}-\ip{\mathcal{Y}^{\vec{\alpha}}\Upsilon_{y^{\vec{\alpha}}}}
{\Upsilon_{y^{\vec{\alpha}}}} \nonumber \\
 & &\hspace{.5in} \mbox{}+\|\Upsilon_{x^{\vec{\alpha}}}\|^{2}\ip{\Upsilon_{x^{\vec{\alpha}}}}{\nabla_0\ln\tp(x^{\vec{\alpha}})}\ln \|\Upsilon_{x^{\vec{\alpha}}}\| \label{contr} \\ 
 & & \hspace{.5in}\mbox{}- \|\Upsilon_{y^{\vec{\alpha}}}\|^{2}\ip{\Upsilon_{y^{\vec{\alpha}}}}{\nabla_0\ln\tp(y^{\vec{\alpha}})}\ln \|\Upsilon_{y^{\vec{\alpha}}}\|\}.\nonumber
\end{eqnarray}
By \cite[Equation 5.6]{B:GS}, we have 
\begin{equation}\label{vectorlimit}
\lim_{\alpha_n\to \infty} 
\lim_{\alpha_{n-1}\to \infty} 
\cdots \lim_{\alpha_2\to \infty} \lim_{\alpha_1\to \infty} \|\Upsilon_{y_{\vec{\alpha}}}\|^2-\|\Upsilon_{x_{\vec{\alpha}}}\|^2=0.
\end{equation}
Also by \cite[Section 3]{B:GC}, we have 
\begin{equation*}
\lim_{\alpha_n\to \infty} \lim_{\alpha_{n-1}\to \infty} \cdots \lim_{\alpha_2\to \infty} \lim_{\alpha_1\to \infty} \ip{\mathcal{X}^{\vec{\alpha}}\Upsilon_{x^{\vec{\alpha}}}}{\Upsilon_{x^{\vec{\alpha}}}}-\ip{\mathcal{Y}^{\vec{\alpha}}\Upsilon_{y^{\vec{\alpha}}}}{\Upsilon_{y^{\vec{\alpha}}}}=0.
\end{equation*}
We therefore turn our attention to 
\begin{eqnarray*}
\|\Upsilon_{x^{\vec{\alpha}}}\|^{2}\ip{\Upsilon_{x^{\vec{\alpha}}}}{\nabla_0\ln\tp(x^{\vec{\alpha}})}\ln \|\Upsilon_{x^{\vec{\alpha}}}\|-\|\Upsilon_{y^{\vec{\alpha}}}\|^{2}\ip{\Upsilon_{y^{\vec{\alpha}}}}{\nabla_0\ln\tp(y^{\vec{\alpha}})}\ln \|\Upsilon_{y^{\vec{\alpha}}}\|.
\end{eqnarray*}
We begin by expanding this term:
\begin{eqnarray*}
\lefteqn{\|\Upsilon_{x^{\vec{\alpha}}}\|^{2}\ip{\Upsilon_{x^{\vec{\alpha}}}}{\nabla_0\ln\tp(x^{\vec{\alpha}})}\ln \|\Upsilon_{x^{\vec{\alpha}}}\|-\|\Upsilon_{y^{\vec{\alpha}}}\|^{2}\ip{\Upsilon_{y^{\vec{\alpha}}}}{\nabla_0\ln\tp(y^{\vec{\alpha}})}\ln \|\Upsilon_{y^{\vec{\alpha}}}\|}& & \\ 
& = & \|\Upsilon_{x^{\vec{\alpha}}}\|^{2}\ip{\Upsilon_{x^{\vec{\alpha}}}}{\nabla_0\ln\tp(x^{\vec{\alpha}})}\bigg(\ln \|\Upsilon_{x^{\vec{\alpha}}}\|-\ln \|\Upsilon_{y^{\vec{\alpha}}}\|\bigg) \\
& + &\bigg(\|\Upsilon_{x^{\vec{\alpha}}}\|^{2}- \|\Upsilon_{y^{\vec{\alpha}}}\|^{2}\bigg)\ip{\Upsilon_{x^{\vec{\alpha}}}}{\nabla_0\ln\tp(x^{\vec{\alpha}})}\ln \|\Upsilon_{y^{\vec{\alpha}}}\| \\
 & + & \|\Upsilon_{y^{\vec{\alpha}}}\|^{2}\bigg(\ip{\Upsilon_{x^{\vec{\alpha}}}}{\nabla_0\ln\tp(x^{\vec{\alpha}})}-
 \ip{\Upsilon_{y^{\vec{\alpha}}}}{\nabla_0\ln\tp(x^{\vec{\alpha}})}\bigg)\ln \|\Upsilon_{y^{\vec{\alpha}}}\|\\
 & + & \|\Upsilon_{y^{\vec{\alpha}}}\|^{2}\bigg(\ip{\Upsilon_{y^{\vec{\alpha}}}}{\nabla_0\ln\tp(x^{\vec{\alpha}})}-
\ip{\Upsilon_{y^{\vec{\alpha}}}}{\nabla_0\ln\tp(y^{\vec{\alpha}})}\bigg)\ln \|\Upsilon_{y^{\vec{\alpha}}}\| \\ 
\end{eqnarray*}
Using  the fact that $1<\tp(x)\in C^1(\Omega)\cap W^{1,\infty}(\Omega)$ and the fact that $\ln a^2=2\ln a$, we have the absolute value of these terms is controlled by a finite constant $C$ times 
\begin{eqnarray}\label{sumest}
\lefteqn{\mathcal{T}\equiv\|\Upsilon_{x^{\vec{\alpha}}}\|^{3}
\bigg|\ln \frac{\|\Upsilon_{x^{\vec{\alpha}}}\|^2}{\|\Upsilon_{y^{\vec{\alpha}}}\|^2}\bigg|  +  \|\Upsilon_{x^{\vec{\alpha}}}\| \bigg|\|\Upsilon_{x^{\vec{\alpha}}}\|^{2}- \|\Upsilon_{y^{\vec{\alpha}}}\|^{2}\bigg|  \Big|\ln \|\Upsilon_{y^{\vec{\alpha}}}\|\Big|}& & \nonumber \\
 & + & \|\Upsilon_{y^{\vec{\alpha}}}\|^{2}\|\Upsilon_{x^{\vec{\alpha}}}-\Upsilon_{y^{\vec{\alpha}}}\|\Big|\ln \|\Upsilon_{y^{\vec{\alpha}}}\|\Big|\\ &  + &  
\|\Upsilon_{y^{\vec{\alpha}}}\|^{3}\|\nabla_0\ln\tp(x^{\vec{\alpha}})-\nabla_0\ln\tp(y^{\vec{\alpha}})\|\Big|\ln \|\Upsilon_{y^{\vec{\alpha}}}\|\Big|.\nonumber
\end{eqnarray}

We will need to consider several cases. These cases rely on the fact that $v$ is a strict supersolution, so that $\|\Upsilon_{y^{\vec{\alpha}}}\|^2>\varepsilon>0$. 

\noindent\textbf{Case 1:} $$\Big|\ln \|\Upsilon_{y^{\vec{\alpha}}}\|\Big|>\Big|\ln \varepsilon\Big|\ \textmd{\ and\ }\ \ln \frac{\|\Upsilon_{x^{\vec{\alpha}}}\|^2}{\|\Upsilon_{y^{\vec{\alpha}}}\|^2}>0.$$

Using the hypotheses of Case 1, we may express $\mathcal{T}$ in terms of coordinates. Namely, 
\begin{eqnarray*}
\lefteqn{\mathcal{T}\equiv\sum_{i=1}^n\bigg(\alpha^2_i\rho^2_i(x^{\vec{\alpha}}_1,x^{\vec{\alpha}}_2,\ldots, x^{\vec{\alpha}}_{i-1})(x^{\vec{\alpha}}_i-y^{\vec{\alpha}}_i)^2\bigg)^{\frac{3}{2}}} & & \\
& \times & \ln \Bigg(1+\frac{\D \sum_{i=2}^n\alpha^2_i\bigg(\rho^2_i(x^{\vec{\alpha}}_1,x^{\vec{\alpha}}_2,\ldots, x^{\vec{\alpha}}_{i-1})-
 \rho^2_i(y^{\vec{\alpha}}_1,y^{\vec{\alpha}}_2,\ldots, y^{\vec{\alpha}}_{i-1})\bigg)(x^{\vec{\alpha}}_i-y^{\vec{\alpha}}_i)^2}{\D \varepsilon}\Bigg) \\
 & &\mbox{}+\bigg(\sum_{i=1}^n \alpha^2_i\rho^2_i(x^{\vec{\alpha}}_1,x^{\vec{\alpha}}_2,\ldots, x^{\vec{\alpha}}_{i-1}) (x^{\vec{\alpha}}_i-y^{\vec{\alpha}}_i)^2\bigg)^{\frac{1}{2}}\\ & \times &  \bigg|\sum_{i=2}^n\alpha^2_i \Big(\rho^2_i(x^{\vec{\alpha}}_1,x^{\vec{\alpha}}_2,\ldots, x^{\vec{\alpha}}_{i-1})-\rho^2_i(y^{\vec{\alpha}}_1,y^{\vec{\alpha}}_2,\ldots, y^{\vec{\alpha}}_{i-1})\Big)(x^{\vec{\alpha}}_i-y^{\vec{\alpha}}_i)^2\bigg| \\
& \times & \frac{1}{2}\Big|\ln \sum_{i=1}^n \alpha^2_i\rho^2_i(y^{\vec{\alpha}}_1,y^{\vec{\alpha}}_2,\ldots, y^{\vec{\alpha}}_{i-1}) (x^{\vec{\alpha}}_i-y^{\vec{\alpha}}_i)^2\Big| \\ & + &
 \sum_{i=1}^n\bigg(\alpha^2_i\rho^2_i(y^{\vec{\alpha}}_1,y^{\vec{\alpha}}_2,\ldots, y^{\vec{\alpha}}_{i-1})(x^{\vec{\alpha}}_i-y^{\vec{\alpha}}_i)^2\bigg) \\
& \times & \sum_{i=2}^n\alpha^2_i\bigg(\rho_i(x^{\vec{\alpha}}_1,x^{\vec{\alpha}}_2,\ldots, x^{\vec{\alpha}}_{i-1})-
 \rho_i(y^{\vec{\alpha}}_1,y^{\vec{\alpha}}_2,\ldots, y^{\vec{\alpha}}_{i-1})\bigg)^2 (x^{\vec{\alpha}}_i-y^{\vec{\alpha}}_i)^2\\ 
 & \times & \frac{1}{2}\Big|\ln \sum_{i=1}^n \alpha^2_i\rho^2_i(y^{\vec{\alpha}}_1,y^{\vec{\alpha}}_2,\ldots, y^{\vec{\alpha}}_{i-1}) (x^{\vec{\alpha}}_i-y^{\vec{\alpha}}_i)^2\Big| \\
 & + & \bigg(\sum_{i=1}^n\Big(\alpha^2_i\rho^2_i(y^{\vec{\alpha}}_1,y^{\vec{\alpha}}_2,\ldots, y^{\vec{\alpha}}_{i-1})(x^{\vec{\alpha}}_i-y^{\vec{\alpha}}_i)^2\Big)\bigg)^{\frac{3}{2}} \\
 & \times & \bigg(\sum_{i=1}^n  \Big(\rho_i(x^{\vec{\alpha}}_1,x^{\vec{\alpha}}_2,\ldots, x^{\vec{\alpha}}_{i-1})\frac{\partial \tp(x)}{\partial x_i}|_{(x^{\vec{\alpha}}_1,x^{\vec{\alpha}}_2,\ldots, x^{\vec{\alpha}}_i)}\\ 
& & \hspace{.5in} \mbox{}-\rho_i(y^{\vec{\alpha}}_1,y^{\vec{\alpha}}_2,\ldots, y^{\vec{\alpha}}_{i-1})\frac{\partial \tp(y)}{\partial y_i}|_{(y^{\vec{\alpha}}_1,y^{\vec{\alpha}}_2,\ldots, y^{\vec{\alpha}}_i)}\Big)^2\bigg)^{\frac{1}{2}} \\
 & \times & \frac{1}{2}\Big|\ln \sum_{i=1}^n \alpha^2_i\rho^2_i(y^{\vec{\alpha}}_1,y^{\vec{\alpha}}_2,\ldots, y^{\vec{\alpha}}_{i-1}) (x^{\vec{\alpha}}_i-y^{\vec{\alpha}}_i)^2\Big|.
\end{eqnarray*}
Note that some of the sums start at $i=2$ since $\rho_1\equiv 1$. 

Using Corollary \ref{lipcor}, Remark \ref{remarkcomp}, and the fact that $\rho_1\equiv 1$ and $\tp(x)\in C^1\cap W^{1,\infty}$, we have for some finite constant $K$, 
 \begin{eqnarray*}
 \lefteqn{0\leq\mathcal{T}_1=\lim_{\alpha_1\to\infty}\mathcal{T}\leq \Bigg(K+\sum_{i=2}^n\bigg(\alpha^2_i\rho^2_i(x^0_1,x^{\vec{\alpha}}_2,\ldots, x^{\vec{\alpha}}_{i-1})(x^{\vec{\alpha}}_i-y^{\vec{\alpha}}_i)^2\bigg)^{\frac{3}{2}}\Bigg)} & & \\
& \times & \ln \Bigg(1+\frac{\D \sum_{i=3}^n\alpha^2_i\bigg(\rho^2_i(x^0_1,x^{\vec{\alpha}}_2,\ldots, x^{\vec{\alpha}}_{i-1})-
 \rho^2_i(x^0_1,y^{\vec{\alpha}}_2,\ldots, y^{\vec{\alpha}}_{i-1})\bigg)(x^{\vec{\alpha}}_i-y^{\vec{\alpha}}_i)^2}{\D \varepsilon}\Bigg) \\
 & &\mbox{}+\bigg(K+\sum_{i=2}^n \alpha^2_i\rho^2_i(x^0_1,x^{\vec{\alpha}}_2,\ldots, x^{\vec{\alpha}}_{i-1}) (x^{\vec{\alpha}}_i-y^{\vec{\alpha}}_i)^2\bigg)^{\frac{1}{2}}
\\ & \times &  \bigg|\sum_{i=3}^n\alpha^2_i \Big(\rho^2_i(x^0_1,x^{\vec{\alpha}}_2,\ldots, x^{\vec{\alpha}}_{i-1})-\rho^2_i(x^0_1,y^{\vec{\alpha}}_2,\ldots, y^{\vec{\alpha}}_{i-1})\Big)(x^{\vec{\alpha}}_i-y^{\vec{\alpha}}_{i-1})^2\bigg| \\
& \times & \frac{1}{2}\Big|\ln \big[K+ \sum_{i=2}^n \alpha^2_i\rho^2_i(x^0_1,y^{\vec{\alpha}}_2,\ldots, y^{\vec{\alpha}}_{i-1}) (x^{\vec{\alpha}}_i-y^{\vec{\alpha}}_i)^2\big]\Big| \\
& + & \Bigg(K+
 \sum_{i=2}^n\bigg(\alpha^2_i\rho^2_i(x^0_1,y^{\vec{\alpha}}_2,\ldots, y^{\vec{\alpha}}_{i-1})(x^{\vec{\alpha}}_i-y^{\vec{\alpha}}_i)^2\bigg)\Bigg) \\
& \times & \sum_{i=3}^n\alpha^2_i\bigg(\rho_i(x^0_1,x^{\vec{\alpha}}_2,\ldots, x^{\vec{\alpha}}_{i-1})-
 \rho_i(x^0_1,y^{\vec{\alpha}}_2,\ldots, y^{\vec{\alpha}}_{i-1})\bigg)^2 (x^{\vec{\alpha}}_i-y^{\vec{\alpha}}_i)^2\\ 
 & \times & \frac{1}{2}\Big|\ln \big[K+ \sum_{i=2}^n \alpha^2_i\rho^2_i(x^0_1,y^{\vec{\alpha}}_2,\ldots, y^{\vec{\alpha}}_{i-1}) (x^{\vec{\alpha}}_i-y^{\vec{\alpha}}_i)^2\big]\Big| \\
 & + & \bigg(K+\sum_{i=2}^n\Big(\alpha^2_i\rho^2_i(x^0_1,y^{\vec{\alpha}}_2,\ldots, y^{\vec{\alpha}}_{i-1})(x^{\vec{\alpha}}_i-y^{\vec{\alpha}}_i)^2\Big)\bigg)^{\frac{3}{2}} \\
 & \times & \bigg(\sum_{i=1}^n  \Big(\rho_i(x^0_1,x^{\vec{\alpha}}_2,\ldots, x^{\vec{\alpha}}_{i-1})\frac{\partial \tp(x)}{\partial x_i}|_{(x^0_1,x^{\vec{\alpha}}_2,\ldots, x^{\vec{\alpha}}_i)}\\ 
& & \hspace{.5in} \mbox{}-\rho_i(x^0_1,y^{\vec{\alpha}}_2,\ldots, y^{\vec{\alpha}}_{i-1})\frac{\partial \tp(y)}{\partial y_i}|_{(x^0_1,y^{\vec{\alpha}}_2,\ldots, y^{\vec{\alpha}}_i)}\Big)^2\bigg)^{\frac{1}{2}} \\
& \times & \frac{1}{2}\Big|\ln \big[K+ \sum_{i=2}^n \alpha^2_i\rho^2_i(x^0_1,y^{\vec{\alpha}}_2,\ldots, y^{\vec{\alpha}}_{i-1}) (x^{\vec{\alpha}}_i-y^{\vec{\alpha}}_i)^2\big]\Big|.
\end{eqnarray*}
Now, if $\rho_2(x^0_1)=0$, the corresponding term vanishes. If $\rho_2(x^0_1)\neq 0$, 
we again apply Corollary \ref{lipcor} and Remark \ref{remarkcomp} to obtain the existence of a finite constant $K$ so that 
 \begin{eqnarray*}
 \lefteqn{0\leq\mathcal{T}_2=\lim_{\alpha_2\to\infty}\lim_{\alpha_1\to\infty}\mathcal{T}\leq \Bigg(K+\sum_{i=3}^n\bigg(\alpha^2_i\rho^2_i(x^0_1,x^0_2,\ldots, x^{\vec{\alpha}}_{i-1})(x^{\vec{\alpha}}_i-y^{\vec{\alpha}}_i)^2\bigg)^{\frac{3}{2}}\Bigg)} & & \\
& \times & \ln \Bigg(1+\frac{\D \sum_{i=4}^n\alpha^2_i\bigg(\rho^2_i(x^0_1,x^0_2,\ldots, x^{\vec{\alpha}}_{i-1})-
 \rho^2_i(x^0_1,x^0_2,\ldots, y^{\vec{\alpha}}_{i-1})\bigg)(x^{\vec{\alpha}}_i-y^{\vec{\alpha}}_{i-1})^2}{\D \varepsilon}\Bigg) \\
 & &\mbox{}+\bigg(K+\sum_{i=3}^n \alpha^2_i\rho^2_i(x^0_1,x^0_2,\ldots, x^{\vec{\alpha}}_{i-1}) (x^{\vec{\alpha}}_i-y^{\vec{\alpha}}_i)^2\bigg)^{\frac{1}{2}}
\\ & \times &  \bigg|\sum_{i=4}^n\alpha^2_i \Big(\rho^2_i(x^0_1,x^0_2,\ldots, x^{\vec{\alpha}}_{i-1})-\rho^2_i(x^0_1,x^0_2,\ldots, y^{\vec{\alpha}}_{i-1})\Big)(x^{\vec{\alpha}}_i-y^{\vec{\alpha}}_i)^2\bigg| \\
& \times & \frac{1}{2}\Big|\ln \big[K+ \sum_{i=3}^n \alpha^2_i\rho^2_i(x^0_1,x^0_2,\ldots, y^{\vec{\alpha}}_{i-1}) (x^{\vec{\alpha}}_i-y^{\vec{\alpha}}_i)^2\big]\Big| \\
& + & \Bigg(K+
 \sum_{i=3}^n\bigg(\alpha^2_i\rho^2_i(x^0_1,x^0_2,\ldots, y^{\vec{\alpha}}_{i-1})(x^{\vec{\alpha}}_i-y^{\vec{\alpha}}_i)^2\bigg)\Bigg) \\
& \times & \sum_{i=4}^n\alpha^2_i\bigg(\rho_i(x^0_1,x^0_2,\ldots, x^{\vec{\alpha}}_{i-1})-
 \rho_i(x^0_1,x^0_2,\ldots, y^{\vec{\alpha}}_{i-1})\bigg)^2 (x^{\vec{\alpha}}_i-y^{\vec{\alpha}}_{i-1})^2\\ 
 & \times & \frac{1}{2}\Big|\ln \big[K+ \sum_{i=3}^n \alpha^2_i\rho^2_i(x^0_1,x^0_2,\ldots, y^{\vec{\alpha}}_{i-1}) (x^{\vec{\alpha}}_i-y^{\vec{\alpha}}_i)^2\big]\Big| \\
 & + & \bigg(K+\sum_{i=3}^n\Big(\alpha^2_i\rho^2_i(x^0_1,x^0_2,\ldots, y^{\vec{\alpha}}_{i-1})(x^{\vec{\alpha}}_i-y^{\vec{\alpha}}_i)^2\Big)\bigg)^{\frac{3}{2}} \\
 & \times & \bigg(\sum_{i=1}^n  \Big(\rho_i(x^0_1,x^0_2,\ldots, x^{\vec{\alpha}}_{i-1})\frac{\partial \tp(x)}{\partial x_i}|_{(x^0_1,x^0_2,\ldots, x^{\vec{\alpha}}_i)}\\ 
& & \hspace{.5in} \mbox{}-\rho_i(x^0_1,x^0_2,\ldots, y^{\vec{\alpha}}_{i-1})\frac{\partial \tp(y)}{\partial y_i}|_{(x^0_1,x^0_2,\ldots, y^{\vec{\alpha}}_i)}\Big)^2\bigg)^{\frac{1}{2}} \\
& \times & \frac{1}{2}\Big|\ln \big[K+ \sum_{i=3}^n \alpha^2_i\rho^2_i(x^0_1,x^0_2,\ldots, y^{\vec{\alpha}}_{i-1}) (x^{\vec{\alpha}}_i-y^{\vec{\alpha}}_i)^2\big]\Big|.
\end{eqnarray*}
 We iterate this process until we arrive at 
 \begin{eqnarray*}
 \lefteqn{0\leq\mathcal{T}_{n-1}=\lim_{\alpha_{n-1}\to\infty}\lim_{\alpha_{n-2}\to\infty}\cdots\lim_{\alpha_1\to\infty}\mathcal{T}} & & \\
 & \leq & \Bigg(K+\bigg(\alpha^2_n\rho^2_n(x^0_1,x^0_2,\ldots, x^0_{n-1})(x^{\vec{\alpha}}_n-y^{\vec{\alpha}}_n)^2 \bigg)^{\frac{3}{2}}\Bigg)
 \times  \ln \Bigg(1+\frac{\D 0}{\D \varepsilon}\Bigg) \\
 & &\mbox{}+\bigg(K+\alpha^2_n\rho^2_n(x^0_1,x^0_2,\ldots, x^0_{n-1}) (x^{\vec{\alpha}}_n-y^{\vec{\alpha}}_n)^2\bigg)^{\frac{1}{2}}
\\ & \times &  \bigg|\alpha^2_n \Big(0\Big)(x^{\vec{\alpha}}_n-y^{\vec{\alpha}}_n)^2\bigg| \\
& \times & \frac{1}{2}\Big|\ln \big[K+ \alpha^2_n\rho^2_n(x^0_1,x^0_2,\ldots, x^0_{n-1}) (x^{\vec{\alpha}}_n-y^{\vec{\alpha}}_n)^2\big]\Big| \\
& + & \Bigg(K+\bigg(\alpha^2_n\rho^2_n(x^0_1,x^0_2,\ldots, x^0_{n-1})(x^{\vec{\alpha}}_n-y^{\vec{\alpha}}_n)^2\bigg)\Bigg) \\
& \times & \alpha^2_n\bigg(0\bigg)^2 (x^{\vec{\alpha}}_n-y^{\vec{\alpha}}_n)^2\\ 
 & \times & \frac{1}{2}\Big|\ln \big[K+ \alpha^2_n\rho^2_n(x^0_1,x^0_2,\ldots, x^0_{n-1}) (x^{\vec{\alpha}}_n-y^{\vec{\alpha}}_n)^2\big]\Big| \\
 & + & \bigg(K+\Big(\alpha^2_n\rho^2_i(x^0_1,x^0_2,\ldots, x^0_{n-1})(x^{\vec{\alpha}}_n-y^{\vec{\alpha}}_n)^2\Big)\bigg)^{\frac{3}{2}} \\
 & \times & \bigg(\sum_{i=1}^n  \Big(\rho_i(x^0_1,x^0_2,\ldots, x^0_{n-1})\frac{\partial \tp(x)}{\partial x_i}|_{(x^0_1,x^0_2,\ldots, x^0_i)}\\ 
& & \hspace{.5in} \mbox{}-\rho_i(x^0_1,x^0_2,\ldots, x^0_{n-1})\frac{\partial \tp(y)}{\partial y_i}|_{(x^0_1,x^0_2,\ldots, x^0_n)}\Big)^2\bigg)^{\frac{1}{2}} \\
& \times & \frac{1}{2}\Big|\ln \big[K+ \alpha^2_n\rho^2_n(x^0_1,x^0_2,\ldots, x^0_{n-1}) (x^{\vec{\alpha}}_n-y^{\vec{\alpha}}_n)^2\big]\Big| \\
& = & 0+0+0+0.
\end{eqnarray*} We then conclude 
 $$\lim_{\alpha_n\to\infty}\lim_{\alpha_{n-1}\to\infty}\cdots\lim_{\alpha_2\to\infty}\lim_{\alpha_1\to\infty}\mathcal{T}=0.$$
 
\noindent\textbf{Case 2:} $$\Big|\ln \|\Upsilon_{y^{\vec{\alpha}}}\|\Big|\leq\Big|\ln \varepsilon\Big|\ \textmd{\ and\ }\ \ln \frac{\|\Upsilon_{x^{\vec{\alpha}}}\|^2}{\|\Upsilon_{y^{\vec{\alpha}}}\|^2}>0.$$
We then have 
\begin{eqnarray*}
\lefteqn{0\leq\mathcal{T}\leq\|\Upsilon_{x^{\vec{\alpha}}}\|^{3}
\bigg|\ln \frac{\|\Upsilon_{x^{\vec{\alpha}}}\|^2}{\|\Upsilon_{y^{\vec{\alpha}}}\|^2}\bigg|  +  \|\Upsilon_{x^{\vec{\alpha}}}\| \bigg|\|\Upsilon_{x^{\vec{\alpha}}}\|^{2}- \|\Upsilon_{y^{\vec{\alpha}}}\|^{2}\bigg|  \Big|\ln \varepsilon\Big|}& &  \\
 & + & \|\Upsilon_{y^{\vec{\alpha}}}\|^{2}\|\Upsilon_{x^{\vec{\alpha}}}-\Upsilon_{y^{\vec{\alpha}}}\|\Big|\ln \varepsilon\Big|\\ &  + &  
\|\Upsilon_{y^{\vec{\alpha}}}\|^{3}\|\nabla_0\ln\tp(x^{\vec{\alpha}})-\nabla_0\ln\tp(y^{\vec{\alpha}})\|\Big|\ln \varepsilon\Big|.
\end{eqnarray*}
We then proceed as in Case 1.
 
\noindent\textbf{Case 3:} $$\Big|\ln \|\Upsilon_{y^{\vec{\alpha}}}\|\Big|>\Big|\ln \varepsilon\Big|\ \textmd{\ and\ }\ \ln \frac{\|\Upsilon_{x^{\vec{\alpha}}}\|^2}{\|\Upsilon_{y^{\vec{\alpha}}}\|^2}<0.$$
From these hypotheses, we have 
\begin{eqnarray*}
\lefteqn{\mathcal{T}=\|\Upsilon_{x^{\vec{\alpha}}}\|^{3}
\ln \frac{\|\Upsilon_{y^{\vec{\alpha}}}\|^2}{\|\Upsilon_{x^{\vec{\alpha}}}\|^2}  +  \|\Upsilon_{x^{\vec{\alpha}}}\| \bigg|\|\Upsilon_{x^{\vec{\alpha}}}\|^{2}- \|\Upsilon_{y^{\vec{\alpha}}}\|^{2}\bigg|  \Big|\ln \|\Upsilon_{y^{\vec{\alpha}}}\|\Big|}& &  \\
 & + & \|\Upsilon_{y^{\vec{\alpha}}}\|^{2}\|\Upsilon_{x^{\vec{\alpha}}}-\Upsilon_{y^{\vec{\alpha}}}\|\Big|\ln \|\Upsilon_{y^{\vec{\alpha}}}\|\Big|\\ &  + &  
\|\Upsilon_{y^{\vec{\alpha}}}\|^{3}\|\nabla_0\ln\tp(x^{\vec{\alpha}})-\nabla_0\ln\tp(y^{\vec{\alpha}})\|\Big|\ln \|\Upsilon_{y^{\vec{\alpha}}}\|\Big|.
\end{eqnarray*}
We will use the following Lemma.
\begin{lemma}
Given that $\|\Upsilon_{y^{\vec{\alpha}}}\|^2>\varepsilon$, we have
$$\|\Upsilon_{x^{\vec{\alpha}}}\|^2 \geq \frac{\varepsilon}{2}.$$
\end{lemma}
\begin{proof}
Suppose not. Then $$\|\Upsilon_{x^{\vec{\alpha}}}\|^2 < \frac{\varepsilon}{2}.$$
We then have
\begin{equation*}
0<\frac{\varepsilon}{2}\leq\|\Upsilon_{y^{\vec{\alpha}}}\|^2-\|\Upsilon_{x^{\vec{\alpha}}}\|^2
\end{equation*}
Taking iterated limits of this inequality contradicts Equation \eqref{vectorlimit}.
\end{proof}
Using this lemma, we have 
\begin{eqnarray*}
\lefteqn{0\leq\mathcal{T}\leq\|\Upsilon_{x^{\vec{\alpha}}}\|^{3}
\ln \Bigg(1+\frac{\|\Upsilon_{y^{\vec{\alpha}}}\|^2-\|\Upsilon_{x^{\vec{\alpha}}}\|^2} {\frac{\varepsilon}{2}}\Bigg)} & &\\  &  + &  \|\Upsilon_{x^{\vec{\alpha}}}\| \bigg|\|\Upsilon_{x^{\vec{\alpha}}}\|^{2}- \|\Upsilon_{y^{\vec{\alpha}}}\|^{2}\bigg|  \Big|\ln \|\Upsilon_{y^{\vec{\alpha}}}\|\Big| \\
& + & \|\Upsilon_{y^{\vec{\alpha}}}\|^{2}\|\Upsilon_{x^{\vec{\alpha}}}-\Upsilon_{y^{\vec{\alpha}}}\|\Big|\ln \|\Upsilon_{y^{\vec{\alpha}}}\|\Big|\\ &  + &  
\|\Upsilon_{y^{\vec{\alpha}}}\|^{3}\|\nabla_0\ln\tp(x^{\vec{\alpha}})-\nabla_0\ln\tp(y^{\vec{\alpha}})\|\Big|\ln \|\Upsilon_{y^{\vec{\alpha}}}\|\Big|.
\end{eqnarray*}
This case then proceeds as in Case 1. 

\noindent\textbf{Case 4:} $$\Big|\ln \|\Upsilon_{y^{\vec{\alpha}}}\|\Big|\leq\Big|\ln \varepsilon\Big|\ \textmd{\ and\ }\ \ln \frac{\|\Upsilon_{x^{\vec{\alpha}}}\|^2}{\|\Upsilon_{y^{\vec{\alpha}}}\|^2}<0.$$
This case is similar to Case 3 (cf. Case 2) and omitted. 

Equation \eqref{contr} now produces a contradiction and the theorem then follows. 
\end{proof}
An analogous argument produces the following Corollary.
\begin{corollary}\label{compinfcor}
Let $v=u_\infty$ be the viscosity solution from Theorem \ref{exist} to 
\begin{equation}\label{maxequ}
\max \{\varepsilon-\|\nabla_0 u\|^2, - \Delta_{\infty(x)} u \} =0
\end{equation}
in a bounded domain $\Omega$. If $u$ is an lower semi-continuous viscosity supersolution to Equation \eqref{maxequ} in $\Omega$ so that $u\geq v$ on $\partial \Omega$, then $u\geq v$ in $\Omega$. 
\end{corollary}

\section{A Harnack Inequality}
We include a Harnack inequality for completeness. First, we have the following lemma whose proof is identical to \cite[Lemma 4.1]{LL} and omitted. 
\begin{lemma}
Let $u$ be a positive viscosity $\infty(x)$-harmonic function and $\zeta$ a positive, compactly supported smooth function. Then 
$$\sup_{x\in\Omega}\bigg|\nabla_0\zeta(x)\nabla_0\ln u(x)\bigg|^{\tp(x)} \leq \sup_{x_\in\Omega}\bigg|\nabla_0\zeta(x)+\zeta(x)\ln\bigg(\frac{\zeta(x)}{u(x)}\bigg)\nabla_0\ln \tp(x)\bigg|^{\tp(x)}.$$
\end{lemma}
As in \cite[Section 4]{LL}, we have the following Harnack inequality as a consequence. 
\begin{thm}
Let $u$ be a positive viscosity $\infty(x)$-harmonic function. Let $B_r$ be a ball of radius $r>0$ contained in the bounded domain $\Omega$. Let $B_{2r}$ be the concentric ball of twice the radius also contained in $\Omega$. Then 
$$\sup_{x\in B_r} u(x) \leq C \big(\inf_{x\in B_r} u(x)+r\big)$$ for some constant $C$ 	depending on $\sup_{x\in B_{2r}}u(x)$. 
\end{thm}


\begin{thebibliography}{BBM}
\bibitem{BR:SRG}Bella\" {\i}che, Andr\' {e}. The Tangent Space in
Sub-Riemannian Geometry. In \emph{Sub-Riemannian Geometry};
Bella\" {\i}che, Andr\' {e}., Risler, Jean-Jacques., Eds.; 
Progress in Mathematics;  Birkh\" {a}user: Basel, Switzerland.
1996; Vol. 144, 1--78.
\bibitem{B:HG}Bieske, Thomas. On Infinite Harmonic Functions on
the Heisenberg Group. Comm. in PDE. \textbf{2002}, \emph{27} (3\&4), 727--762.
\bibitem{B:GS}Bieske, Thomas. Lipschitz Extensions in Grushin-type Spaces. Mich Math J. \emph{2005}, 53:1, 3--31. 
\bibitem{B:GC}Bieske, Thomas. Properties of infinite harmonic functions in Grushin-type spaces. Rocky Mtn J. of Math. \emph{2009}, 39:3, 729--756.
\bibitem{B:MP}Bieske, Thomas. A Sub-Riemannian Maximum Principle and its application to the $p$-Laplacian in Carnot Groups. Ann. Acad. Sci. Fenn. Math. \emph{2012}, 37:1, 1--16. 
 \bibitem{CC}Calin, Ovidiu; Chang, Der-Chen Sub-Riemannian geometry. General theory and examples. Encyclopedia of Mathematics and its Applications, 126. Cambridge University Press, Cambridge, 2009.
\bibitem{CLR}Chen, Yunmei; Levine, Stacey; Rao, Murali.
Variable exponent, linear growth functionals in image restoration. 
SIAM J. Appl. Math. \textbf{2006}, \emph{66} (4), 1383--1406.
\bibitem{Je:ULE} Jensen, Robert.
Uniqueness of Lipschitz Extensions:  Minimizing the Sup Norm of the
Gradient. Arch. Ration. Mech. Anal. \textbf{1993}, \emph{123}, 51--74.
\bibitem{LL}Lindqvist, Peter; Lukkari, Teemu A curious equation involving the $\infty$-Laplacian. Adv. Calc. Var. 3 \textbf{2010}, \emph{3} (4), 409--421.
\bibitem{R}R\r{u}\v{z}i\v{c}ka, Michael. Electrorheological fluids: modeling and mathematical theory. Lecture Notes in Mathematics, 1748. Springer-Verlag, Berlin, 2000. 
\end{thebibliography}
\end{document}